\newcommand{\autosyncline}[1]{}
	\providecommand{\cvsId}{}

\providecommand{\cvsId}{\\\small\texttt{Id: nfpoly.tex,v 1.42 2008-08-20 18:18:21 jdemeyer Exp }}
\documentclass[12pt,a4paper,english]{article}
\newcommand{\Author}{Jeroen Demeyer\thanks{The author is a Postdoctoral Fellow of the Research Foundation --- Flanders (FWO).
\textbf{Address:} Ghent University, Department of Pure Mathematics and Computer Algebra, Krijgslaan 281, 9000 Gent, Belgium.
\textbf{E-mail:} jdemeyer@cage.ugent.be.
}}
\newcommand{\Title}{Diophantine sets of polynomials over number fields}

\providecommand{\cvsId}{\\\small\texttt{Id: header.tex,v 1.53 2008-07-11 21:18:29 jdemeyer Exp }}
\usepackage{amsmath,amsfonts,amsthm,amssymb}
\usepackage{color}
\usepackage{ifthen}
\usepackage{enumerate}
\usepackage{picins}
\usepackage[mathlines,displaymath]{lineno}  

	\usepackage[pdftex]{graphics}
	\usepackage[pdftex,pdfstartview={FitH},bookmarks]{hyperref}
\newcommand{\twodigit}[1]{\ifthenelse{#1 < 10}{0#1}{#1}}
\providecommand{\Date}{\number\year--\twodigit{\number\month}--\twodigit{\number\day}}
\providecommand{\cvsId}{}
\author{\Author}
\title{\Title}
\date{\Date\cvsId}

\providecommand{\theoremname}{Theorem}
\providecommand{\maintheoremname}{Main Theorem}
\providecommand{\conjname}{Conjecture}
\providecommand{\openname}{Open Problem}
\providecommand{\questionname}{Question}
\providecommand{\propname}{Proposition}
\providecommand{\obsname}{Observation}
\providecommand{\lemmaname}{Lemma}
\providecommand{\mainlemmaname}{Main Lemma}
\providecommand{\corname}{Corollary}
\providecommand{\claimname}{Claim}
\providecommand{\definename}{Definition}
\providecommand{\examplename}{Example}
\providecommand{\cexamplename}{Counterexample}
\providecommand{\exercisename}{Exercise}
\providecommand{\factname}{Fact}
\providecommand{\factsname}{Facts}
\providecommand{\remarkname}{Remark}
\providecommand{\solutionname}{Solution}
\providecommand{\stepname}{Step}

\theoremstyle{plain}
\ifx\chaptername\undefined
	\newtheorem{theorem}{\theoremname}
\else
	\newtheorem{theorem}{\theoremname}[chapter]
\fi
\newtheorem*{theorem*}{\theoremname}

\newtheorem*{mtheorem*}{\maintheoremname}

\newtheorem*{conj*}{\conjname}

\newtheorem*{open*}{\openname}
\newtheorem{prop}[theorem]{\propname}
\newtheorem*{prop*}{\propname}

\newtheorem*{obs*}{\obsname}
\newtheorem{lemma}[theorem]{\lemmaname}
\newtheorem*{lemma*}{\lemmaname}

\newtheorem*{mlemma*}{\mainlemmaname}
\newtheorem{cor}[theorem]{\corname}
\newtheorem*{cor*}{\corname}

\theoremstyle{definition}
\newtheorem{define}[theorem]{\definename}
\newtheorem*{define*}{\definename}

\newtheorem*{example*}{\examplename}

\newtheorem*{cexample*}{\cexamplename}

\newtheorem*{question*}{\questionname}

\newtheorem*{exercise*}{\exercisename}

\newtheorem*{fact*}{\factname}

\newtheorem*{facts*}{\factsname}

\theoremstyle{remark}

\newcounter{step}

\newcounter{case}

\newcommand{\inv}{^{-1}}

\newcommand{\frakp}{\mathfrak{p}}

\newcommand{\C}{\mathbb{C}}

\newcommand{\N}{\mathbb{N}}
\newcommand{\Q}{\mathbb{Q}}

\newcommand{\Z}{\mathbb{Z}}

\newcommand{\calC}{\mathcal{C}}

\newcommand{\calG}{\mathcal{G}}

\newcommand{\calO}{\mathcal{O}}

\newcommand{\calR}{\mathcal{R}}
\newcommand{\calS}{\mathcal{S}}
\newcommand{\calT}{\mathcal{T}}
\newcommand{\calU}{\mathcal{U}}

\newcommand{\chebX}{\mathrm{X}}
\newcommand{\chebY}{\mathrm{Y}}

\newcommand{\AS}[2][void]{\ifthenelse{\equal{#1}{void}}{\mathbb{A}\!^{#2}}{\mathbb{A}\!^{#2}(#1)}}
\newcommand{\PS}[2][void]{\ifthenelse{\equal{#1}{void}}{\mathbb{P}^{#2}}{\mathbb{P}^{#2}(#1)}}

\newcommand{\FF}[1]{\mathbb{F}\!_{#1}}

\renewcommand{\vec}{\overrightarrow}

\DeclareMathOperator{\lcm}{lcm}

\newcommand{\set}[2]{\{{#1}\mid{#2}\}}

\newcommand{\form}[1]{\langle{#1}\rangle}

\newcommand{\into}{\hookrightarrow}

\ifx\needsha\undefined
\else
	\DeclareFontEncoding{OT2}{}{}

\fi

\begin{document}

\maketitle

\begin{abstract}
\autosyncline{22\\nfpoly}Let $\calR$ be a recursive subring of a number field.
We show that recursively enumerable sets
are diophantine for the polynomial ring $\calR[Z]$.
\end{abstract}


\section{Introduction}\label{sec-intro}

\autosyncline{42\\nfpoly}Let $\calR$ be a recursive subring of a number field.
In this paper, we show that recursively enumerable (r.e.)\ subsets of $\calR[Z]^k$ are diophantine.

\autosyncline{45}For any recursively stable integral domain,
one can easily see that every diophantine set is recursively enumerable
(see the end of section~\ref{sec-definitions}).
However, the converse problem
--- are recursively enumerable sets diophantine? --- is much more difficult.

\autosyncline{51}In 1970, Matiyasevich (\cite{matiy-h10}) showed,
building on earlier work by Davis, Putnam and Robinson,
that r.e.\ sets are diophantine for the integers $\Z$.
This had as an immediate consequence the negative answer to Hilbert's Tenth Problem:
there exists no algorithm which can decide whether
\autosyncline{56}a diophantine equation over $\Z$ has a zero over $\Z$.
See \cite{davis-h10} for a good write-up
of the various steps in the proof that r.e.\ sets are diophantine for $\Z$,
and hence the negative answer to Hilbert's Tenth Problem.

\autosyncline{61}The undecidability of diophantine equations has been shown for many other rings and fields,
\cite{pheidas-zahidi} and \cite{poonen-notices} give a good overview of what is known.
On the other hand, the equivalence of r.e.\ and diophantine sets
is much stronger and much less is known.

\autosyncline{66}Apart from the original result for $\Z$,
this equivalence has been shown for $\Z[Z]$ by Denef (\cite{denef-zt}),
for $\calO_K[Z_1, \dots, Z_n]$ where $K$ is a totally real number field by Zahidi
(\cite{zahidi-poly} and \cite{zahidi-thesis}).
In characteristic $p$, it is known for $\FF{q}[Z]$ and for $K[Z]$
\autosyncline{71}where $K$ is a recursive algebraic extension of a finite field by the author (\cite{demeyer-ffpoly2}).
The latter ring is not recursively stable,
so the equivalence is between diophantine sets
and sets which are r.e.\ for every recursive presentation.
All these results use the fact that r.e.\ sets are diophantine for $\Z$.
\autosyncline{76}This paper is no exception,
however we base ourselves on Denef's result for $\Z[Z]$.

\subsection{Definitions}\label{sec-definitions}

\autosyncline{83\\nfpoly}We quickly recall the definitions of recursively enumerable sets,
recursive rings and diophantine sets.
For more information, we refer to the introductory texts
\cite{poonen-notices} and \cite{pheidas-zahidi}.

\begin{define*}
\autosyncline{89}Let $\calS$ be a subset of $\N^k$.
Then $\calS$ is called \emph{recursively enumerable} (r.e.)\ if
there exists an algorithm which prints out elements of $\calS$ as it runs,
such that all elements of $\calS$ are eventually printed at least once.
Since $\calS$ can be infinite, this algorithm is allowed to run infinitely long
\autosyncline{94}and use an unbounded amount of memory.
\end{define*}

\autosyncline{97}Since there are only countably many algorithms but uncountably many
subsets of $\N^k$, there certainly exist sets which are not recursively enumerable.
There also exist sets which are recursively enumerable but whose complement is not.
Finite unions, finite intersections, cartesian products and projections $\N^{k+r} \to \N^k$
of recursively enumerable sets are still recursively enumerable.

\begin{define*}
\autosyncline{104}Let $\calR$ be a countable ring.
Then $\calR$ is called a \emph{recursive ring} if
there exists a bijection $\theta: \calR \to \N$
such that the sets
$$
	\set{(\theta(X), \theta(Y), \theta(X+Y)) \in \N^3}{X,Y \in \calR}
	\text{ and }
	\set{(\theta(X), \theta(Y), \theta(XY)) \in \N^3}{X,Y \in \calR}
$$ \autosyncline{112}are recursively enumerable.
In this case, $\theta$ is called a \emph{recursive presentation} of $\calR$.
A recursive ring $\calR$ is called \emph{recursively stable}
if for any two recursive presentations $\theta_1$ and $\theta_2$,
the set $\set{(\theta_1(X), \theta_2(X)) \in \N^2}{X \in \calR}$ is recursively enumerable.
\end{define*}

\autosyncline{119}The intuition of a recursive ring is a ring in which we can effectively compute,
it is a ring whose elements can be represented by a computer.
The recursive presentation $\theta$ gives every element of $\calR$ a ``code'',
such that, given the codes of $X$ and $Y$, we can compute the code of $X+Y$ and of $XY$.
If we have two different recursive presentations $\theta_1$ and $\theta_2$,
\autosyncline{124}then an element $X$ of $\calR$ has two ``codes'' $\theta_1(X)$ and $\theta_2(X)$.
A ring is recursively stable if and only if
$\theta_2(X)$ can be effectively computed from $\theta_1(X)$.

\begin{define*}
\autosyncline{130}Let $\calR$ be a recursively stable ring
with a recursive presentation $\theta: \calR \to \N$.
Then a subset $\calS \subseteq \calR^k$ is called \emph{recursively enumerable}
if and only if $\theta^{\otimes k}(\calS)$ is an r.e.\ subset of $\N^k$.
\end{define*}

\autosyncline{136}Intuitively, we can still think of r.e.\ subsets of $\calR^k$
as sets which can be printed by an algorithm (possibly running infinitely long).
The requirement that $\calR$ is recursively stable
implies that the definition of r.e.\ subsets of $\calR^k$
does not depend on the choice of $\theta$.
\autosyncline{141}One can prove (see \cite{froh-shep}) that every
field which is finitely generated over its prime field is recursively stable.
Furthermore, a recursive integral domain with a recursively stable fraction field
is automatically recursively stable.
Since we assumed that $\calR$ was recursive we have that $\calR$
\autosyncline{146}is recursively stable, hence $\calR[Z]$ is recursively stable.
To construct an example of a ring which is not recursive,
consider any non-r.e.\ subset $\calS$ of $\N$.
Now take the localization of $\Z$ where the $n$-th prime number
is inverted if and only if $n \in \calS$.
\autosyncline{151}This is a non-recursive subring of $\Q$.

\begin{define*}
\autosyncline{154}Let $\calR$ be an integral domain and $\calS$ a subset of $\calR^k$.
Then $\calS$ is called \emph{diophantine}
if there exists a polynomial $p(a_1, \dots, a_k, x_1, \dots, x_n)$ with coefficients in $\calR$
such that
\begin{equation}\label{define-dioph}
	\calS = \set{(a_1, \dots, a_k) \in \calR^k}{p(a_1, \dots, a_k, x_1, \dots, x_n) = 0 \text{ for some } x_1, \dots, x_n \in \calR}
.\end{equation}
\autosyncline{161}The polynomial $p$ is called a \emph{diophantine definition} of $\calS$.
A function $f: \calR^m \to \calR^n$ is called \emph{diophantine}
if the set $\set{(\vec{X}, f(\vec{X})) \in \calR^{m+n}}{\vec{X} \in \calR^m}$ is diophantine.
\end{define*}

\autosyncline{166}When dealing with decidability questions (analogues of Hilbert's Tenth Problem)
it often makes sense to restrict the coefficients of the polynomial $p$ to a subring of $\calR$.
This is certainly necessary if $\calR$ is uncountable.
However, if we want to prove that r.e.\ sets are diophantine,
then every singleton in $\calR$ needs to be diophantine.
\autosyncline{171}Therefore, we might as well assume that we take all of $\calR$ as ring of coefficients.

\autosyncline{173}If $\calR$ is a recursively stable ring,
then every diophantine set is recursively enumerable.
To see this, consider a diophantine set $\calS$ defined as in \eqref{define-dioph}.
Construct an algorithm which tries all possible values
$(a_1, \dots, a_k, x_1, \dots, x_n) \in \calR^{k+n}$ and evaluates $p(a_1, \dots, a_k, x_1, \dots, x_n)$.
\autosyncline{178}Whenever is zero is found, it prints $(a_1, \dots, a_k)$.
This algorithm will print exactly the set $\calS$.

\subsection{Overview}\label{sec-overview}

\autosyncline{185\\nfpoly}Let $K$ be a number field
and let $\calR$ be a subring of $K$ with fraction field $K$.
\autosyncline{191}In order to prove that r.e.\ sets are diophantine for $\calR[Z]$,
the main result is the following from section~\ref{sec-intpol}:
\begin{theorem*}
Let $\calR$ be a noetherian integral domain of characteristic zero
such that the degree function $\calR[Z] \setminus \{0\} \to \Z$ is diophantine.
\autosyncline{196}Then $\Z[Z]$ is a diophantine subset of $\calR[Z]$.
\end{theorem*}

\autosyncline{199}To prove this, we first show that the set of polynomials in $\calR[Z]$
which divide some $Z^u - 1$ is diophantine.
This is done using a Pell equation,
similarly to the definition of powers of $Z$ in \cite{demeyer-ffpoly}, Section 4.
A polynomial $F$ dividing $Z^u - 1$, normalised such that $F(0) = 1$,
\autosyncline{204}has coefficients in $\Z$ 
if and only if $F(h) \in \Z$ for a sufficiently large number $h$
(depending only on the degree of $F$).
In this way, we diophantinely define the polynomials in $\Z[Z]$ dividing some $Z^u - 1$.
We call these the \emph{root-of-unity polynomials}.
\autosyncline{209}This set is $Z$-adically dense in $\Z[[Z]]^*$,
which allows us to diophantinely define all of $\Z[Z]$ in $\calR[Z]$.

\autosyncline{212}Once we have a diophantine definition of $\Z[Z]$,
it follows from \cite{denef-zt} that r.e.\ subsets of $\Z[Z]^k$ are diophantine over $\calR[Z]$.
\autosyncline{217}From this, the main result for $\calR[Z]$ easily follows.

\autosyncline{219}At several points in the proof above we need a diophantine definition
of the degree function $\deg: \calR[Z] \setminus \{0\} \to \Z$.
We give such a diophantine definition in section~\ref{sec-degree}.
We apply a result by Kim and Roush who showed in \cite{kim-roush-padic}
that diophantine equations over $L(Z)$ are undecidable
\autosyncline{224}if $L$ is contained in a finite extension of $\Q_p$ for some $p \geq 3$.
They showed undecidability
by giving a diophantine definition of the discrete valuation ring $L[Z]_{(Z)}$.
Since ``negative degree'' is a discrete valuation,
the same method gives a diophantine definition of ``degree'' in $\calR[Z]$.

\section{Special polynomials}\label{sec-special}

\autosyncline{243\\nfpoly}In this section, we state some properties of the Chebyshev polynomials $\chebX_n$ and $\chebY_n$
and cyclotomic polynomials $\Phi_n$.
We also define root-of-unity polynomials.
Everything in this section concerns only the ring $\Z[Z]$.

\subsection{Chebyshev polynomials}\label{sec-cheb}

\begin{define}\label{cheb}
\autosyncline{253\\nfpoly}Let $n \in \Z$ and define polynomials $\chebX_n, \chebY_n \in \Z[Z]$ using
the following equality:
\begin{equation}\label{define-cheb}
	(Z + \sqrt{Z^2-1})^n = \chebX_n(Z) + \sqrt{Z^2-1} \; \chebY_n(Z)
.\end{equation}
\autosyncline{258}Since $(Z + \sqrt{Z^2-1})\inv = (Z - \sqrt{Z^2-1})$, this definition makes sense for negative $n$.
\end{define}

\autosyncline{261}The degree of $\chebX_n$ is $|n|$; the degree of $\chebY_n$ is $|n|-1$ for $n \neq 0$,
while $\chebY_0 = 0$.

\autosyncline{264}In the literature, $\chebX_n$ is called the $n$-th Chebyshev polynomial of the first kind
and $\chebY_{n+1}$ is called the $n$-th Chebyshev polynomial of the second kind
(such that the $n$-th Chebyshev polynomials have degree $n$ for $n \geq 0$).

\autosyncline{268}The couples $(\chebX_n, \chebY_n)$ satisfy the Pell equation $X^2 - (Z^2-1) Y^2 = 1$.
Conversely, we have:
\begin{prop}\label{cheb-pell}
Let $\calR$ be an integral domain of characteristic zero
and $T$ a non-constant polynomial in $\calR[Z]$.
\autosyncline{273}If $X$ and $Y$ in $\calR[Z]$ satisfy $X^2 - (T^2-1) Y^2 = 1$,
then $X = \pm \chebX_n(T)$ and $Y = \chebY_n(T)$ for some $n \in \Z$.
\end{prop}

\begin{proof}
\autosyncline{278}See \cite{denef-ffpoly}, Lemma~2.1.
Since $\chebX_{-n} = \chebX_n$ and $\chebY_{-n} = -\chebY_n$,
we do not need to put $\pm$ in front of $\chebY_n(T)$.
\end{proof}

\autosyncline{283}The Chebyshev polynomials also satisfy the following identity:
\begin{prop}\label{cheb-power}
In $\Q(Z)$, the following equality holds for all $n \in \Z$:
\begin{equation}\label{cheb-power-eq}
    Z^n = \chebX_n\left(\frac{Z + Z\inv}{2}\right)
	    + \frac{Z - Z\inv}{2} \chebY_n\left(\frac{Z + Z\inv}{2}\right)
.\end{equation}
\end{prop}

\begin{proof}
\autosyncline{293}This easily follows from \eqref{define-cheb}.
\end{proof}

\subsection{Cyclotomic and root-of-unity polynomials}\label{sec-cyclo}

\autosyncline{300\\nfpoly}Let $\Phi_n \in \Z[Z]$ denote the $n$-th cyclotomic polynomial.

\begin{prop}\label{cyclo-term}
\autosyncline{303}Let $n \geq 2$ and write $n = \prod_{i=1}^k p_i^{e_i}$,
where the $p_i$'s are distinct primes and every $e_i \geq 1$.
Let $d := \prod_{i=1}^k p_i^{e_i-1}$.
Then
$$
	\Phi_n(Z) \equiv 1 + (-1)^{k+1} Z^d \pmod{Z^{2d}}
.$$
\end{prop}

\begin{proof}
\autosyncline{313}If $\mu$ denotes the M\"obius function, then we have
$$
	\Phi_n(Z) = \prod_{a|n} (Z^{n/a} - 1)^{\mu(a)}
.$$
Since $n \geq 2$, we have $\sum_{a|n} \mu(a) = 0$
\autosyncline{318}and we can multiply by $1 = \prod_{a|n} (-1)^{\mu(a)}$
to get:
$$
	\Phi_n(Z) = \prod_{a|n} (1 - Z^{n/a})^{\mu(a)}
.$$
\autosyncline{323}Now we evaluate this product modulo $Z^{2d}$.

\autosyncline{325}If $n/a \geq 2d$ then $(1 - Z^{n/a})^{\mu(a)}$ is congruent to $1 \pmod{Z^{2d}}$.
The same happens if $a$ is not squarefree since in this case $\mu(a) = 0$.
The only squarefree $a$ dividing $n$ such that $n/a < 2d$ equals $a = n/d$.
So we have
$$
	\Phi_n(Z) \equiv (1 - Z^d)^{\mu(n/d)} \pmod{Z^{2d}}
.$$
\autosyncline{332}If $k$ is even, then $\mu(n/d) = 1$ and we have the desired result.
If $k$ is odd, then $\mu(n/d) = -1$ and we have
$(1 - Z^d)\inv = (1 + Z^d)(1 - Z^{2d})\inv \equiv 1 + Z^d \pmod{Z^{2d}}$.
\end{proof}

\begin{cor}\label{cyclo-construct}
\autosyncline{338}Let $d \in \N$ and $s \in \{-1,1\}$.
Then there exist infinitely many $n \in \N$ such that
$$
	\Phi_n(Z) \equiv 1 + s Z^d \pmod{Z^{2d}}
.$$
\end{cor}

\begin{proof}
\autosyncline{346}Write $d := \prod_{i=1}^k p_i^{e_i}$
and let $m := \prod_{i=1}^k p_i^{e_i+1}$.
If $r$ is any squarefree number coprime to $m$,
then it follows from Proposition~\ref{cyclo-term}
that $\Phi_{rm}(Z)$ is congruent to $1 \pm Z^d \pmod{Z^{2d}}$,
\autosyncline{351}where the sign of $Z^d$ is determined by the parity of the number of factors in $rm$.
Now the statement clearly follows.
\end{proof}

\begin{define}
\autosyncline{357}We call a polynomial $F \in \Z[Z]$ a \emph{root-of-unity polynomial}
if it satisfies one of the following three equivalent conditions:
\begin{enumerate}\setlength{\itemsep}{0em}
\item $F$ is a divisor of $Z^u - 1$ for some $u > 0$.
\item $F$ or $-F$ is a product of distinct cyclotomic polynomials.
\item $F(0) = \pm 1$, $F$ \autosyncline{362}is squarefree and all the zeros of $F$ are roots of unity.
\end{enumerate}
\end{define}

\autosyncline{366}Let $\calC$ denote the set of all root-of-unity polynomials,
and let $\calC^{+}$ denote those with constant term equal to $1$.

\begin{prop}\label{cyclo-approx}
\autosyncline{371}Let $F \in \Z[Z]$ with $F(0) \in \{-1,1\}$, and let $d \in \Z_{>0}$.
Then there exists a polynomial $M \in \calC$ such that $F \equiv M \pmod{Z^d}$.
\end{prop}

\autosyncline{375}If we are working in the $Z$-adic topology,
then ``$F \equiv M \pmod{Z^d}$'' means that $M$ is an approximation of $F$
with a precision of $Z^d$.
Since the units of $\Z[[Z]]$ are exactly the power series $F$ with $F(0) = \pm 1$,
the proposition can be rephrased as follows:
\emph{the set of root-of-unity polynomials is $Z$-adically dense in $\Z[[Z]]^*$.}

\begin{proof}
\autosyncline{383}Since the set $\calC$ is invariant under changing sign,
we may assume without loss of generality that $F(0) = 1$.

\autosyncline{386}The proof will be done by induction on $d$,
which means that we will construct better and better approximations of $F$.
For $d = 1$, we can take $M = 1$.
Now let $d \geq 1$ and assume that $F \equiv M_0 \pmod{Z^d}$, where $M_0 \in \calC$.
Then $F - M_0 \equiv c Z^d \pmod{Z^{d+1}}$ for some $c \in \Z$.
\autosyncline{391}If $c$ happens to be zero, then we can take $M = M_0$.

\autosyncline{393}First consider the case $c > 0$.
By Corollary~\ref{cyclo-construct}, we can find an $n_1 \in \N$
such that $\Phi_{n_1}(Z) \equiv 1 + Z^d \pmod{Z^{2d}}$
and such that $\Phi_{n_1}(Z)$ is not a factor of $M_0$.
Let $M_1 := M_0 \Phi_{n_1}(Z)$.
\autosyncline{398}Since $M_0(0) = 1$, we get
$$
	F - M_1 \equiv F - M_0(1 + Z^d) \equiv (F - M_0) - M_0 Z^d \equiv (c-1) Z^d \pmod{Z^{d+1}}
.$$
We can iterate this procedure.
\autosyncline{403}Set $M_2 := M_1 \Phi_{n_2}(Z)$ for a $\Phi_{n_2}$ which is congruent to $1 + Z^d \pmod{Z^{2d}}$,
then $F - M_2 \equiv (c-2) Z^d \pmod{Z^{d+1}}$.
After $c$ steps, we have $F - M_c \equiv 0 \pmod{Z^{d+1}}$.
So we can take $M := M_c$.

\autosyncline{408}The case $c < 0$ is analogous,
the only difference is that we need to multiply
with polynomials which are congruent to $1 - Z^d \pmod{Z^{d+1}}$.
\end{proof}

\section{Defining polynomials with integer coefficients}\label{sec-intpol}

\autosyncline{417\\nfpoly}Throughout this section,
$\calR$ is a noetherian integral domain of characteristic zero
such that the degree function $\calR[Z] \setminus \{0\} \to \Z$ is diophantine.
If $\calR$ is a subring of a number field,
it is a noetherian integral domain of characteristic zero
\autosyncline{422}and in section~\ref{sec-degree} we will show that ``degree'' is diophantine for such $\calR[Z]$.
When we say that ``degree'' is diophantine, we actually mean that the composition
$\calR[Z] \setminus \{0\} \to \Z \into \calR[Z]$ is diophantine.
This makes sense since the set $\Z$ is diophantine in $\calR[Z]$
(see \cite{shlapentokh-poly}, Theorem~5.1).

\autosyncline{428}In this section, we show that $\Z[Z]$ is a diophantine subset of $\calR[Z]$.
This is done in three steps:
first we diophantinely define all divisors of some $Z^u - 1$ in $\calR[Z]$.
Second, we restrict these to the polynomials which have integer coefficients,
i.e.\ the root-of-unity polynomials.
\autosyncline{433}Third, we use Proposition~\ref{cyclo-approx} to get all of $\Z[Z]$ in $\calR[Z]$.

\subsection{Divisors of $Z^u - 1$}

\autosyncline{439\\nfpoly}We give a diophantine definition of the divisors of $Z^u - 1$,
without requiring that they have coefficients in $\Z$.
For technical reasons, we first restrict ourselves to polynomials of degree at least $3$.

\begin{lemma}\label{cyclo-div3}
\autosyncline{444}For $G \in \calR[Z]$ with $\deg(G) \geq 3$, we have
\begin{align}
	(\exists u > 0)&(G \mid Z^u - 1 ~\wedge~ G(0) = 1) \label{cyclo-div3-top}\\
	&\Updownarrow \notag\\
	(\exists S,X,Y)(&X^2 - \left({\left(\textstyle\frac{Z + S}{2}\right)^2} - 1\right)Y^2 = 1 ~\wedge~ X \equiv 1 \pmod{Z + S - 2} \label{cyclo-div3-pell}\\
		\wedge~ &Y \neq 0 ~\wedge~ G = 1 - Z S ~\wedge~ X + \left(\textstyle\frac{Z - S}{2}\right) Y \equiv 1 \pmod{G}) \label{cyclo-div3-pow}
\end{align}
\end{lemma}

\begin{proof}
\autosyncline{454}The formula $(\exists S)(G = 1 - Z S)$ is equivalent to $G(0) = 1$.
Since $\deg(G) \geq 3$ and $G = 1 - Z S$, it follows that $\deg(S) \geq 2$.
Therefore $Z + S$ is non-constant.
By Proposition \ref{cheb-pell}, the first part of formula \eqref{cyclo-div3-pell} is equivalent to
$$
	X = \pm \chebX_n\!\left(\textstyle\frac{Z + S}{2}\right)
	\text{ and } 
	Y = \chebY_n\!\left(\textstyle\frac{Z + S}{2}\right)
	\text{ for some $n \in \Z$}
.$$
\autosyncline{464}Since $X_n(1) = 1$, the condition $X \equiv 1 \pmod{Z + S - 2}$
forces the sign of $X$ to be positive.
The formula $Y \neq 0$ is equivalent to $n \neq 0$.

\autosyncline{468}In the last part of formula \eqref{cyclo-div3-pow},
we are working modulo $G = 1 - Z S$.
But this means that $S \equiv Z\inv \pmod{G}$.
So, that formula becomes equivalent to
$$
	\chebX_n\!\left(\textstyle\frac{Z + Z\inv}{2}\right)
	+ \left(\textstyle\frac{Z - Z\inv}{2}\right)
	  \chebY_n\!\left(\textstyle\frac{Z + Z\inv}{2}\right)
	\equiv 1 \pmod{G}
.$$
\autosyncline{478}Using Proposition~\ref{cheb-power}, this is equivalent to $Z^n \equiv 1 \pmod{G}$.
Without loss of generality, we may assume that $n \geq 0$
(otherwise multiply both sides by $Z^{-n}$).
Then we can rewrite $Z^n \equiv 1 \pmod{G}$ as $G \mid Z^n - 1$.
\end{proof}

\begin{prop}\label{cyclo-div}
\autosyncline{485}In $\calR[Z]$, the set of all polynomials dividing $Z^u - 1$ for some $u > 0$ is diophantine.
\end{prop}

\begin{proof}
\autosyncline{489}Let $F$ be an element of $\calR[Z]$.
We claim that $F$ divides some $Z^u - 1$ if and only if
\begin{equation}\label{cyclo-div-eq}
	(\exists G)\big(F \mid G ~\wedge~ (Z^3 - 1) \mid G ~\wedge~ (\exists u > 0)(G \mid Z^u - 1 ~\wedge~ G(0) = 1)\big)
.\end{equation}

\autosyncline{495}If formula \eqref{cyclo-div-eq} is satisfied, then $F \mid G \mid Z^u - 1$.
Conversely, if $F \mid Z^u - 1$, we can set $G = \lcm(Z^3 - 1, F)$.
Then $G$ will divide $Z^{3u} - 1$.
Since $F$ divides $Z^u - 1$, its constant coefficient must be a unit,
therefore $G$ can be chosen to have $G(0) = 1$.

\autosyncline{501}Applying Lemma~\ref{cyclo-div3}, we see that \eqref{cyclo-div-eq} is diophantine.
Indeed, a congruence $A \equiv B \pmod{C}$ can be written as $(\exists X)(A - B = C X)$.
The formula $Y \neq 0$ is diophantine using the fact that $\calR[Z]$ is noetherian
(see \cite{mb-nonnul}, Th\'eor\`eme~3.1).
Hence, formulas \eqref{cyclo-div3-pell}--\eqref{cyclo-div3-pow} are diophantine.
\autosyncline{506}We can apply Lemma~\ref{cyclo-div3} because the $G$ appearing in \eqref{cyclo-div-eq}
must have degree $\geq 3$.
\end{proof}

\subsection{Root-of-unity polynomials}

\autosyncline{514\\nfpoly}Now we have a diophantine definition of the divisors of $Z^u - 1$,
but we only want those divisors with integer coefficients.
We take care of this using the following proposition,
which was inspired by \cite{denef-zt} and \cite{zahidi-poly}.

\begin{prop}\label{div-large}
\autosyncline{520}Let $K$ be a number field and $\calO$ its ring of integers.
Let $F \in \calO[Z]$ be a polynomial satisfying $F(0) \in \{-1,1\}$
whose zeros (over an algebraic closure) are all roots of unity.
If $F(2^{\deg F} + 1)$ is an integer, then every coefficient of $F$ is an integer.
\end{prop}

\begin{proof}
\autosyncline{527}By changing sign if necessary, we may assume without loss of generality that $F(0) = 1$.
Let $d$ be the degree of $F$ and write
\begin{equation}\label{div-large-Fsum}
	F(Z) = \sum_{i=0}^d \alpha_i Z^i
,\end{equation} where $\alpha_i \in \calO$.
\autosyncline{532}Note that $\alpha_d \neq 0$ and $\alpha_0 = 1$.

\autosyncline{534}If $d = 0$, then $F(Z) = 1$ which is in $\Z[Z]$.
Now assume that $d \geq 1$.
Over an algebraic closure, $F$ can be factored as
\begin{equation}
	F(Z) = \alpha_d (Z - \zeta_1) \dots (Z - \zeta_d)
,\end{equation} \autosyncline{539}where every $\zeta_i$ is a root of unity.
We see that $F(0) = \alpha_d (-1)^d \prod_{i=1}^d \zeta_i$.
This must be equal to $1$, therefore $\alpha_d$ is also a root of unity.
Write $\sigma_{d,i}$ for the $i$-th elementary symmetric polynomial in $d$ variables.
Since $\sigma_{d,i}$ has $\binom{d}{i}$ terms,
\autosyncline{544}it follows that $\alpha_i = \alpha_d \cdot \sigma_{d,i}(\zeta_1, \dots, \zeta_d)$
is the sum of $\binom{d}{i}$ roots of unity.

\autosyncline{547}Let $|\cdot|$ be an archimedean absolute value on $K$
(i.e.\ an absolute value coming from an embedding $K \into \C$).
Then we have $|\alpha_i| \leq \binom{d}{i}$.
Since $\binom{d}{i} \leq 2^{d-1}$ for all $d \geq 1$ and all $i \in \{0, \dots, d\}$,
we have $|\alpha_i| \leq 2^{d-1}$.

\autosyncline{553}Define the set $\calG_d \subseteq \calO[Z]$ consisting of all polynomials $G \in \calO[Z]$
satisfying:
\begin{enumerate}\setlength{\itemsep}{0em}
\item The degree of $G$ is at most $d$.
\item $G(2^d + 1)$ is an integer.
\item $|\gamma_i| \leq 2^{d-1}$ \autosyncline{558}for every coefficient $\gamma_i$ of $G$
	and every archimedean absolute value on $K$.
\end{enumerate}
Clearly, the elements of $\Z[Z]$ having degree at most $d$
and coefficients in the interval $\{-2^{d-1}, \ldots, 2^{d-1}\}$ are in $\calG_d$.
\autosyncline{563}There are $(2^d+1)^{d+1}$ such polynomials.
We claim that these are the only elements of $\calG_d$.
Since $F$ is in $\calG_d$, this claim implies the proposition.

\autosyncline{567}To prove the claim,
take any $G$ in $\calG_d$ and write $G = \sum_{i=0}^d \gamma_i Z^i$
(where we allow $\gamma_d = 0$).
We have the following bound for all $h \in \Z$ with $h > 1$:
$$
	|G(h)|
		\leq \sum_{i=0}^d |\gamma_i| h^i
		\leq 2^{d-1} \frac{h^{d+1} - 1}{h - 1}
.$$
\autosyncline{576}Fix $h := 2^d + 1$ for the remainder of this proof.
Then we have $|G(h)| \leq (h^{d+1} - 1)/2$.

\autosyncline{579}Now take two elements $G \neq H$ in $\calG_d$ and let $D := G - H$.
Write $D(Z) = \sum_{i=0}^e \delta_i Z^i$ with $\delta_e \neq 0$ (clearly, $e \leq d$).
We want to prove that $D(h) \neq 0$, so assume that $D(h) = 0$.
Then
\begin{equation}\label{div-large-deltae}
	\delta_e h^e = -\sum_{i=0}^{e-1} \delta_i h^i
.\end{equation}
\autosyncline{586}The coefficients of $G$ and $H$ have absolute value at most $2^{d-1}$,
therefore $|\delta_i| \leq 2^d$.
Since $\delta_e \in \calO$ is integral over $\Z$,
we have $|\delta_e|_\frakp \leq 1$ for every non-archimedean ($\frakp$-adic) absolute value on $K$.
From the product formula for absolute values it follows that
$|\delta_e| \geq 1$ \autosyncline{591}for some archimedean absolute value on $K$.
If we take such an absolute value,
then \eqref{div-large-deltae} implies the following contradiction:
$$
	h^e \leq |\delta_e h^e| \leq \sum_{i=0}^{e-1} |\delta_i| h^i \leq 2^d \frac{h^e - 1}{h - 1} = h^e - 1
.$$

\autosyncline{598}Consider again the set $\calG_d$.
We just showed that $G(h)$ cannot take the same value for two different elements $G$ in $\calG_d$.
Since $G(h) \in \Z$ by definition of $\calG_d$ and $|G(h)| \leq (h^{d+1} - 1)/2$,
it follows that $\calG_d$ has at most $h^{d+1}$ elements.
But we already know that there are $h^{d+1}$ elements in $\calG_d \cap \Z[Z]$,
\autosyncline{603}therefore $\calG_d \subseteq \Z[Z]$.
\end{proof}

\autosyncline{607}Taking Propositions~\ref{cyclo-div} and \ref{div-large} together, we can now prove:
\begin{prop}\label{dioph-C}
If ``degree'' is diophantine in $\calR[Z]$,
then the set $\calC$ is a diophantine subset of $\calR[Z]$.
\end{prop}

\begin{proof}
\autosyncline{614}The $\calR[Z]$-divisors of $Z^u - 1$ are diophantine by Proposition~\ref{cyclo-div}.
If we take only those polynomials with $F(0) = \pm 1$,
they satisfy the conditions of Proposition~\ref{div-large}
with $K = \Q(\zeta_u)$ where $\zeta_u$ is a primitive $u$-th root of unity.
Note that $F(0) \in \{-1, 1\}$ is equivalent to $Z \mid F^2 - 1$, a diophantine condition.
\autosyncline{619}The formula
\begin{equation}\label{cyclo-div-deg}
	(\exists t \in \Z)\big(F \equiv t \pmod{Z - 2^{\deg(F)} - 1}\big)
\end{equation}
expresses that $F$ evaluated at $2^{\deg(F)} + 1$ is an integer.
\autosyncline{624}Since $\Z$ is diophantine in $\calR[Z]$ (see \cite{shlapentokh-poly}, Theorem~5.1)
and ``degree'' is diophantine by assumption,
formula \eqref{cyclo-div-deg} is diophantine.
\end{proof}

\subsection{All polynomials with integer coefficients}

\autosyncline{633\\nfpoly}Proposition~\ref{dioph-C} gives
us a diophantine definition of $\calC$, which is a subset of $\Z[Z]$.
To define all of $\Z[Z]$ in $\calR[Z]$, we use Proposition~\ref{cyclo-approx}.
By taking remainders of the elements of $\calC$ after Euclidean division by $Z^d$,
we get all elements of $\Z[Z]$ with constant coefficient $1$ or $-1$.
\autosyncline{638}We don't actually need that the set of powers of $Z$ is diophantine,
we can divide by elements of $\calC + 1$, which contains the powers of $Z$.
In order for Euclidean division to be diophantine,
we need ``degree'' to be diophantine.
To get all elements of $\Z[Z]$,
\autosyncline{643}we just need to add an integer to the polynomials we get as remainders.

\begin{theorem}\label{dioph-ZZ}
\autosyncline{646}Let $\calR$ be a noetherian integral domain of characteristic zero such that ``degree'' is diophantine.
Then $\Z[Z]$ is a diophantine subset of $\calR[Z]$.
\end{theorem}

\begin{proof}
\autosyncline{651}Let $X$ be an element of $\calR[Z]$.
We claim that $X$ is in $\Z[Z]$ if and only if
\begin{align}
	(\exists M,D,Q,R,C)(
		& M \in \calC ~\wedge~ D \in \calC ~\wedge~ (Z-1) \mid D \label{def-ZZ-MD}\\
		\wedge~ & M = Q (D+1) + R ~\wedge~ (R = 0 ~\vee~ \deg(R) < \deg(D)) \label{def-ZZ-R}\\
		\wedge~ & C \in \Z ~\wedge~ X = R + C) \label{def-ZZ-C}
.\end{align}

\autosyncline{660}Assume that $X$ is indeed in $\Z[Z]$.
Then set $C := X(0) - 1$ and $R := X - C$ such that $R(0) = 1$.
Let $D := Z^{\deg(R) + 1} - 1$.
Apply Proposition~\ref{cyclo-approx} to find an $M \in \calC$ such that $R \equiv M \pmod{D+1}$
and let $Q := (M - R)/(D + 1)$.
\autosyncline{665}Now it is clear that \eqref{def-ZZ-MD}--\eqref{def-ZZ-C} is satisfied.

\autosyncline{667}Conversely, assume that \eqref{def-ZZ-MD}--\eqref{def-ZZ-C} is satisfied,
we have to show that $X \in \Z[Z]$.
Since $\calC \subseteq \Z[Z]$, we know that $M$ and $D$ are in $\Z[Z]$.
The condition $Z - 1 \mid D$ prevents $D$ from being constant
(note that $0$ is not an element of $\calC$).
\autosyncline{672}Since all elements of $\calC$ are monic up to sign, $D+1$ is also.
Formula \eqref{def-ZZ-R} says that $R$ is the remainder of
the Euclidean division of $M$ by $D+1$, therefore $R \in \Z[Z]$.
Since $C \in \Z$, it also follows that $X \in \Z[Z]$.
\end{proof}

\section{Diophantine definition of degree}\label{sec-degree}

\autosyncline{682\\nfpoly}We start with a lemma which shows that defining the degree function
$\calR[Z] \setminus \{0\} \to \Z$ is equivalent to defining a certain ``weak'' degree equality relation.

\begin{lemma}\label{deg-equal}
\autosyncline{686}Let $\calR$ be an integral domain of characteristic zero.
Let $\delta(F,X)$ be a diophantine relation on $\calR[Z]^2$
such that $\delta(F,X)$ is equivalent to $\deg(F) = \deg(X)$
for all $F \in \calR[Z] \setminus \{0\}$ and $X \in \Z[Z] \setminus \{0\}$.
Then the relation ``$\deg(F) = d$''
\autosyncline{691}between $F \in \calR[Z] \setminus \{0\}$ and $d \in \Z_{\geq 0}$ is diophantine.
\end{lemma}

\begin{proof}
\autosyncline{695}Let $F \in \calR[Z] \setminus \{0\}$ and let $d \in \Z_{\geq 0}$.
We claim that $F$ has degree $d$ if and only if
\begin{equation}\label{deg-equal-eq}
	(\exists X,Y)(X^2 - (Z^2 - 1)Y^2 = 1 ~\wedge~ Y(1) = d ~\wedge~ \delta(F,X))
.\end{equation}
\autosyncline{700}Since $\delta$ is diophantine and $Y(1) = d$ is equivalent to $Z - 1 \mid Y - d$,
this formula is clearly diophantine.

\autosyncline{703}Assume that \eqref{deg-equal-eq} is satisfied.
Since $\chebY_n(1) = n$ for any $n \in \Z$,
the subformula ``$X^2 - (Z^2 - 1)Y^2 = 1 ~\wedge~ Y(1) = d$'' is equivalent to
``$X = \pm \chebX_d(Z) ~\wedge~ Y = \chebY_d(Z)$'' by Proposition~\ref{cheb-pell}.
In particular, $X$ is an element of $\Z[Z]$ of degree $d$.
\autosyncline{708}By the assumptions on $\delta$, this implies that $\deg(F) = d$.

\autosyncline{710}Conversely, if the degree of $F$ equals $d$,
then we set $X = \chebX_d(Z)$ and $Y = \chebY_d(Z)$.
This satisfies \eqref{deg-equal-eq}.
\end{proof}

\autosyncline{715}As in the Introducion, let $K$ be a number field and
$\calR$ a subring of $K$ with fraction field $K$.

\autosyncline{718}To diophantinely define degree in $\calR[Z]$,
we use the fact that ``negative degree'' is a discrete valuation on $K(Z)$.
More precisely, if $F, G \in \calR[Z]$, then $v_{Z\inv}(F/G) := \deg(G) - \deg(F)$
defines a discrete valuation on $K(Z)$.
Therefore, the problem reduces to showing
\autosyncline{723}that the discrete valuation ring at $Z\inv$ in $K(Z)$ is diophantine.
For this, we need certain quadratic forms used by Kim and Roush (see \cite{kim-roush-padic})
to prove undecidability for rational function fields over so-called $p$-adic fields
with $p$ odd.
This undecidability has been generalised to arbitrary function fields
\autosyncline{728}over $p$-adic fields with $p$ odd (see \cite{mb-ell} or \cite{eisentraeger-padic}).

\begin{define}
\autosyncline{731}Let $p$ be a prime number.
A field $K$ is called \emph{$p$-adic} if $K$
can be embedded in a finite extension of $\Q_p$.
\end{define}

\autosyncline{736}It is clear from this definition that every number field is $p$-adic for every $p$.
For the rest of this section, we fix any odd prime $p$.
Following the method by Kim and Roush,
we need to work over a field satisfying Hypothesis $(\mathcal{H})$.

\begin{define}
\autosyncline{742}Let $L$ be a $p$-adic field with $p$ odd
and let $v_\frakp$ be a discrete valuation on $L$ extending the $p$-adic valuation on $\Q$.
We say that $L$ satisfies Hypothesis $(\mathcal{H})$ if and only if
$L$ contains elements $\alpha$ and $\pi$ such that
\begin{enumerate}\setlength{\itemsep}{0em}
\item $v_\frakp(\pi)$ \autosyncline{747}is odd and $\pi$ is algebraic over $\Q$.
\item $\alpha$ is a root of unity.
\item $L$ contains a square root of $-1$.
\item The quadratic form $\form{1,\alpha} \form{1,\pi}$ is anisotropic
	(i.e.\ has no non-trivial zeros) in the completion $L_\frakp$.
\item \autosyncline{752}The quadratic form $\form{1,\alpha} \form{1,\pi}$ is isotropic
	in all $2$-adic completions of $\Q(\alpha, \pi, \sqrt{-1})$.
\end{enumerate}
\end{define}

\begin{prop}[\cite{kim-roush-padic}, Proposition~8]
\autosyncline{758}Let $K$ be a $p$-adic field for an odd prime $p$.
Then there exists a finite extension $L$ of $K$ which satisfies Hypothesis $(\mathcal{H})$.
\end{prop}

\autosyncline{762}The next two propositions deal with certain quadratic forms.
Our variable $Z$ is the inverse of the variable $t$ that Kim and Roush use.

\begin{prop}[\cite{kim-roush-padic}, Proposition~7]\label{qf-aniso}
\autosyncline{766}Let $L$ be any field of characteristic $0$
and suppose that $\form{1,-\alpha} \form{1,\pi}$ is an anisotropic quadratic form over $L$.
Let $F \in L(Z)$ such that $v_{Z\inv}(F)$ is non-negative and even.
Then one of the following two is anisotropic over $L(Z)$:
\begin{gather}
	\form{Z, -\alpha Z, -1, -F} \form{1,\pi} \\
	\form{Z, -\alpha Z, -1, -\alpha F} \form{1,\pi}
.\end{gather}
\end{prop}

\autosyncline{776}The following proposition follows from \cite{kim-roush-padic}.
However, here we use a reformulation by Eisentr\"ager
(see \cite{eisentraeger-padic}, Theorem~8.1).
Note that the condition that $G$ has algebraic coefficients
is missing from Eisentr\"ager's paper, but it is necessary and it does appear in Kim and Roush.

\begin{prop}\label{qf-iso}
\autosyncline{783}Let $L$ be a $p$-adic field satisfying Hypothesis $(\mathcal{H})$
for elements $\alpha$ and $\pi$ in $L$.
Let $\calU \subseteq L(Z)$ such that $\calU \cap \Q$ is dense in $\Q_{p_1} \times \dots \times \Q_{p_m}$
for every finite set of rational primes $\{p_1, \dots, p_m\}$.
Let $G \in L(Z)$ such that $v_Z(G) = -2$ and $v_{Z\inv}(G) = 1$.
\autosyncline{788}Assume that $G = G_N(Z)/G_D(Z)$ for polynomials $G_N$ and $G_D$
with coefficients algebraic over $\Q$.
Then there exist $\gamma_3, \gamma_5 \in \calU$ such that,
if we let
\begin{equation}\label{qff}
	F := (1 + Z\inv)^3 G(Z) + \gamma_3 Z^{-3} + \gamma_5 Z^{-5}
,\end{equation} \autosyncline{794}then the following quadratic forms are both isotropic over $L(Z)$:
\begin{gather}
	\form{Z, \alpha Z, -1, -F} \form{1,\pi} \label{qf1}\\
	\form{Z, \alpha Z, -1, -\alpha F} \form{1,\pi} \label{qf2}
.\end{gather}
\end{prop}

\autosyncline{801}The most natural choice for $\calU$ would be $\calU = L$.
However, for our applications, $\calU$ needs to be diophantine in $L(Z)$.
In the article by Kim and Roush, $\calU$ is a subset of $L$.
However, since enlarging the set $\calU$ only weakens the proposition,
we can even take $\calU$ in $L(Z)$.

\autosyncline{807}Taking these last two propositions together,
we can prove the following:
\begin{prop}\label{qf-iff}
Let $L$ and $\calU$ be as in Proposition~\ref{qf-iso}
with the additional condition that every element $A \in \calU$ has $v_{Z\inv}(A) \geq 0$.
\autosyncline{812}Let $X \in L(Z)$ with algebraic coefficients and define
$$
	G(Z) := \frac{(Z + Z^2) + X^3}{Z^3 + Z^2 X^3}
.$$
Then $v_{Z\inv}(X) \geq 0$ if and only if there exist $\gamma_3, \gamma_5 \in \calU$
\autosyncline{817}such that the quadratic forms \eqref{qf1} and \eqref{qf2} are both isotropic
with $F$ as in \eqref{qff}.
\end{prop}

\begin{proof}
\autosyncline{822}Write $G_N := (Z + Z^2) + X^3$ and $G_D := Z^3 + Z^2 X^3$
such that $G = G_N/G_D$.
Assume that $v_{Z\inv}(X) \geq 0$.
Then $v_{Z\inv}(G_N) = -2$ and $v_{Z\inv}(G_D) = -3$, such that $v_{Z\inv}(G) = 1$.
If $v_Z(X) \geq 1$, then $v_Z(G_N) = 1$ and $v_Z(G_D) = 3$, such that $v_Z(G) = -2$.
\autosyncline{827}If $v_Z(X) \leq 0$, then $v_Z(G_N) = 3 v_Z(X)$ and $v_Z(G_D) = 2 + 3 v_Z(X)$, such that $v_Z(G) = -2$.
Summarized, if $v_{Z\inv}(X) \geq 0$, then we have $v_{Z\inv}(G) = 1$ and $v_Z(G) = -2$.
Proposition~\ref{qf-iso} gives us that \eqref{qf1} and \eqref{qf2}
are indeed isotropic for some choice of $\gamma_3$ and $\gamma_5$ in $\calU$.

\autosyncline{832}Conversely, assume that $v_{Z\inv}(X) < 0$.
We must show that one of the forms \eqref{qf1} or \eqref{qf2} is anisotropic
for every $\gamma_3, \gamma_5$ with non-negative valuation at $Z\inv$.
Since $v_{Z\inv}(X) \leq -1$, we have $v_{Z\inv}(G_N) = 3 v_{Z\inv}(X)$
and $v_{Z\inv}(G_D) = -2 + 3 v_{Z\inv}(X)$.
\autosyncline{837}Therefore $v_{Z\inv}(G) = 2$.
Since $v_{Z\inv}(\gamma_i) \geq 0$, it follows from \eqref{qff} that $v_{Z\inv}(F) = 2$.
Hypothesis $(\mathcal{H})$ says that $\form{1,\alpha} \form{1,\pi}$
is locally anisotropic at $\frakp$, hence it is also globally anisotropic over $L$.
Since $L$ contains $\sqrt{-1}$, signs in quadratic forms do not matter.
\autosyncline{842}Therefore, we can apply Proposition~\ref{qf-aniso}.
\end{proof}

\begin{theorem}\label{deg-padic}
\autosyncline{846}Let $\calR$ be a subfield of a number field $K$ with fraction field $K$.
In the ring $\calR[Z]$,
the relation ``$\deg(X) = d$''
between $X \in \calR[Z] \setminus \{0\}$ and $d \in \Z_{\geq 0}$ is diophantine.
\end{theorem}

\begin{proof}
\autosyncline{853}Let $X,Y \in \calR[Z] \setminus \{0\}$.
If we can give a diophantine definition of ``$\deg(X) \leq \deg(Y)$'',
then ``$\deg(X) \leq \deg(Y) ~\wedge~ \deg(Y) \leq \deg(X)$''
is a predicate $\delta(X,Y)$ which satisfies the conditions of Lemma~\ref{deg-equal}.

\autosyncline{858}Since the non-zero elements of $\calR[Z]$ form a diophantine subset of $\calR[Z]$
(see \cite{mb-nonnul}),
we can construct a diophantine interpretation of the fraction field $K(Z)$ over $\calR[Z]$.
Let $L$ be a finite extension of $K$ which satisfies Hypothesis $(\mathcal{H})$.
Using a basis of $L$ as a $K$-vector space,
\autosyncline{863}there is a diophantine model of $L(Z)$ over $K(Z)$.

\autosyncline{865}Since $\deg(X) \leq \deg(Y)$ is equivalent to $v_{Z\inv}(X/Y) \geq 0$,
it suffices to give a diophantine definition
of the predicate ``$v_{Z\inv}(X) \geq 0$'' with $X \in L(Z)$.
Let
$$
	\calU = \set{n/P}{n \in \Z ~\wedge~ P \in \calR[Z] \setminus \{0\}} \subseteq K(Z)
.$$
\autosyncline{872}By construction, every element $A \in \calU$ has $v_{Z\inv}(A) \geq 0$.
The set $\calU$ contains $\Q$,
which is clearly dense in every $\Q_{p_1} \times \dots \times \Q_{p_m}$.
Since quadratic forms being isotropic is a diophantine condition
and $\calU$ is diophantine,
\autosyncline{877}it follows by Proposition~\ref{qf-iff} that ``$v_{Z\inv}(X) \geq 0$'' is diophantine.
\end{proof}

\section{Recursively enumerable sets}\label{sec-re}

\autosyncline{884\\nfpoly}In this final section we discuss how having a diophantine definition of $\Z[Z]$
in $\calR[Z]$ gives us that r.e.\ subsets of $\calR[Z]^k$ are diophantine.
Recall that $\calR$ is a subring of a number field $K$ with fraction field $K$.

\autosyncline{888}Denef showed (see \cite{denef-zt}) that
r.e.\ subsets of $\Z[Z]^k$ are diophantine over $\Z[Z]$.
Since we showed in the preceding sections that $\Z[Z]$
is a diophantine subset of $\calR[Z]$,
it also follows that r.e.\ subsets of $\Z[Z]^k$ are diophantine over $\calR[Z]$.

\autosyncline{894}Let $\alpha \in \calR$ such that $K = \Q(\alpha)$ and let $d := [K : \Q]$.
Now any element $X$ of $\calR[Z]$ can be written as
\begin{equation}\label{write-basis}
	X = \frac{X_0 + X_1 \alpha + \dots + X_{d-1} \alpha^{d-1}}{y}
\end{equation}
\autosyncline{899}with $X_i$ in $\Z[Z]$ and $y$ in $\Z \setminus \{0\}$.

\autosyncline{901}Now let $\calS \subseteq \calR[Z]$ be an r.e.\ set,
we have to show that $\calS$ is diophantine.
To $\calS$ we associate a set $\calT \subseteq \Z[Z]^{d+1}$ using \eqref{write-basis}:
the set $\calT$ has one tuple $(X_0, X_1, \dots, X_{d-1}, y) \in \Z[Z]^{d+1}$
for every $X \in \calS$.
\autosyncline{906}This tuple $(X_0, X_1, \dots, X_{d-1}, y)$ is not unique but that is not a problem,
we can for example try all possible tuples and take the first one which works for a given $X$.
This way, we have a bijection between $\calS$ and $\calT$.
Moreover, the set $\calT$ will also be r.e.,\ since we can construct $\calT$ from $\calS$
using a recursive procedure.
\autosyncline{911}Since $\calT$ is a subset of $\Z[Z]^{d+1}$, it will be diophantine over $\calR[Z]$.
Now it immediately follows that $\calS$ is diophantine:
$$
	X \in \calS \iff \big(\exists (X_0, X_1, \dots, X_{d-1}, y) \in \calT\big)
		\big(X y = X_0 + X_1 \alpha + \dots + X_{d-1} \alpha^{d-1}\big)
.$$
\autosyncline{917}The argument for sets $\calS \subseteq \calR[Z]^k$ is very similar,
using a set $\calT \subseteq \Z[Z]^{(d+1)k}$.

\bibliographystyle{amsplain}
\bibliography{all}

\providecommand{\bysame}{\leavevmode\hbox to3em{\hrulefill}\thinspace}
\providecommand{\MR}{\relax\ifhmode\unskip\space\fi MR }
\providecommand{\MRhref}[2]{%
  \href{http://www.ams.org/mathscinet-getitem?mr=#1}{#2}
}
\providecommand{\href}[2]{#2}
\begin{thebibliography}{10}

\bibitem{davis-h10}
Martin Davis, \emph{Hilbert's tenth problem is unsolvable}, Amer.\ Math.\
  Monthly \textbf{80} (1973), no.~3, 233--269.

\bibitem{demeyer-ffpoly}
Jeroen Demeyer, \emph{Recursively enumerable sets of polynomials over a finite
  field}, J.\ Algebra \textbf{310} (2007), no.~2, 801--828.

\bibitem{demeyer-ffpoly2}
\bysame, \emph{Recursively enumerable sets of polynomials over a finite field
  are {D}iophantine}, Invent.\ Math. \textbf{170} (2007), no.~3, 655--670.

\bibitem{denef-zt}
Jan Denef, \emph{Diophantine sets over {$\mathbb{Z}[T]$}}, Proc.\ Amer.\ Math.\
  Soc. \textbf{69} (1978), no.~1, 148--150.

\bibitem{denef-ffpoly}
\bysame, \emph{The {D}iophantine problem for polynomial rings of positive
  characteristic}, Logic Colloquium 78 (M.~Boffa, D.~van Dalen, and K.~Mcaloon,
  eds.), Studies in logic and the foundations of mathematics, no.~97,
  North-Holland, 1979, pp.~131--145.

\bibitem{eisentraeger-padic}
Kirsten Eisentr\"ager, \emph{Hilbert's tenth problem for function fields of
  varieties over number fields and $p$-adic fields}, J.\ Algebra \textbf{310}
  (2007), no.~2, 775--792.

\bibitem{froh-shep}
A.~Fr\"ohlich and C.~Shepherdson, \emph{Effective procedures in field theory},
  Phil.\ Trans.\ Roy.\ Soc.\ London \textbf{248} (1956), 407--432.

\bibitem{kim-roush-padic}
Ki~Hang Kim and Fred Roush, \emph{Diophantine unsolvability over $p$-adic
  function fields}, J.\ Algebra \textbf{176} (1995), no.~1, 83--110.

\bibitem{matiy-h10}
Yuri Matiyasevich, \emph{Enumerable sets are {D}iophantine}, Soviet Math.\
  Dokl. \textbf{11} (1970), 354--358.

\bibitem{mb-ell}
Laurent Moret-Bailly, \emph{Elliptic curves and {H}ilbert's tenth problem for
  algebraic function fields over real and $p$-adic fields}, J.\ Reine und
  Angew.\ Math. \textbf{587} (2005), 77--143.

\bibitem{mb-nonnul}
\bysame, \emph{Sur la d\'efinissabilit\'e existentielle de la non-nullit\'e
  dans les anneaux}, Algebra \& Number Theory \textbf{1} (2007), no.~3,
  331--346.

\bibitem{pheidas-zahidi}
Thanases Pheidas and Karim Zahidi, \emph{Undecidability of existential theories
  of rings and fields: a survey}, Hilbert's Tenth Problem: Relations with
  Arithmetic and Algebraic Geometry ({Ghent, 1999}) (Denef et~al., eds.),
  Contemp.\ Math., vol. 270, 2000, pp.~49--105.

\bibitem{poonen-notices}
Bjorn Poonen, \emph{Undecidability in number theory}, Notices of the AMS
  \textbf{55} (2008), no.~3, 344--350.

\bibitem{shlapentokh-poly}
Alexandra Shlapentokh, \emph{Diophantine definitions for some polynomial
  rings}, Commun.\ Pure Appl.\ Math. \textbf{43} (1990), 1055--1066.

\bibitem{zahidi-thesis}
Karim Zahidi, \emph{Existential undecidability for rings of algebraic
  functions}, Ph.D. thesis, Ghent University, 1999.

\bibitem{zahidi-poly}
\bysame, \emph{On diophantine sets over polynomial rings}, Proc.\ Amer.\ Math.\
  Soc. \textbf{128} (2000), no.~3, 877--884.

\end{thebibliography}

\end{document}